\documentclass{amsart}

\title{Unitary representations of oligomorphic groups}
\author{Todor Tsankov}
\address{Universit\'e Paris 7 \\ UFR de Math\'ematiques, case 7012  \\ 75205 Paris \textsc{cedex} 13}
\email{todor@math.jussieu.fr}
\thanks{Research partially supported by the ANR network AGORA, NT09-461407.}

\usepackage[T1]{fontenc}
\usepackage{amssymb}

\usepackage[bb-bold]{my-macros}
\renewcommand{\Q}{{\mathbf Q}}

\newcommand{\Uspenskii}{Uspenski\xspace}

\usepackage[initials, shortalphabetic]{amsrefs}

\numberwithin{equation}{section}

\newcounter{ExamplesEnum}

\begin{document}

\begin{abstract}
We obtain a complete classification of the continuous unitary representations of oligomorphic permutation groups (those include the infinite permutation group $S_\infty$, the automorphism group of the countable dense linear order, the homeomorphism group of the Cantor space, etc.). Our main result is that all irreducible representations of such groups are obtained by induction from representations of finite quotients of open subgroups and moreover, every representation is a sum of irreducibles. As an application, we prove that many oligomorphic groups have property (T). We also show that the Gelfand--Raikov theorem holds for topological subgroups of $S_\infty$: for all such groups, continuous irreducible representations separate points in the group.
\end{abstract}

\maketitle

\section{Introduction}
Abstract harmonic analysis is classically restricted to studying representations of locally compact groups, and for a good reason: the Haar measure provides an invaluable tool for constructing and analyzing representations. It gives rise to the left-regular representation (so that every locally compact group has at least one faithful representation) but also allows to define the convolution algebra of the group and various useful topologies on function spaces which are important for understanding the representations. And indeed, many standard theorems of the subject break down for non-locally compact groups: for example, the group of orientation-preserving homeomorphisms of the reals has no non-trivial unitary representations whatsoever (Megrelishvili~\cite{Megrelishvili2001}), while the group of all measurable maps from $[0, 1]$ to the circle has a faithful unitary representation by multiplication on $L^2([0, 1])$ but no irreducible representations (this example is due to Glasner~\cite{Glasner1998a}; see also \cite{Bekka2008}*{Example~C.5.10}). Those rather pathological examples suggest that any attempt to develop a representation theory for non-locally compact groups should be restricted to certain well-behaved classes.

And indeed, a number of interesting classification results have been obtained for the representations of some concrete non-locally compact groups. Lieberman~\cite{Lieberman1972} classified the unitary dual of the infinite symmetric group $S_\infty$ (for us, $S_\infty$ is always the group of \emph{all} permutations of the integers and not only those with finite support). Later, Olshanski developed a rather versatile machine, called the \df{semigroup method}, using which he succeeded to give another proof of Lieberman's theorem \cite{Olshanski1985} and also classify the representations of many other non-locally compact groups including the infinite-dimensional unitary, orthogonal, and symplectic groups, and (a variant of) the infinite-dimensional general linear groups over finite fields \cite{Olshanski1991a} (cf. Remark~\ref{r:Olshanski}). See the survey \cite{Olshanski1991a} for an explanation of the method and the references therein for more examples. A classification for the representations of the unitary group had also been announced by Kirillov~\cite{Kirillov1973}. All of the groups above have only countably many irreducible representations and every representation splits as a sum of irreducibles, a situation that very much resembles the case of compact groups.

In this paper, we study the unitary representations of \df{non-archimedean}, separable groups (i.e. separable groups that admit a countable basis at the identity consisting of open subgroups), also known as subgroups of $S_\infty$. This property alone allows us to recover one important result valid in the locally compact situation, namely the Gelfand--Raikov theorem. More precisely, we prove the following.
\begin{theorem} \label{th:GelfRaikov}
Let $G$ be a topological subgroup of $S_\infty$. Then for every $x, y \in G$, $x \neq y$, there exists a continuous, irreducible, unitary representation $\pi$ of $G$ such that $\pi(x) \neq \pi(y)$.
\end{theorem}

The main body of the paper, however, concentrates on groups that have an additional special property, that of \df{Roelcke precompactness} (cf. Definition~\ref{df:Roelcke}). It turns out that this property has a natural translation in the language of permutation groups and model theory. For us, a \df{permutation group} will be a topological subgroup of the group of all permutations of a countable set $\bX$. The following definition is of central importance for this paper.
\begin{defn} \label{df:olig}
Let $\bX$ be a countable (finite or infinite) set. A permutation group $G \actson \bX$ is called \df{oligomorphic} if the diagonal action $G \actson \bX^n$ has only finitely many orbits for each $n$. A topological group $G$ is \df{oligomorphic} if it can be realized as an oligomorphic permutation group.
\end{defn}

Closed oligomorphic permutation groups also have a model-theoretic interpretation: they are exactly the automorphism groups of $\omega$-categorical structures (cf. Section~\ref{s:olig-groups}). A standard way to produce $\omega$-categorical structures is the \Fraisse construction: given a class of finite structures satisfying a certain amalgamation property, there is a way to build a (unique) infinite, homogeneous structure that contains all structures in the class as substructures. We postpone the formal definitions to Section~\ref{s:olig-groups} and just describe a few examples.

\smallskip \noindent \textbf{Examples of $\omega$-categorical \Fraisse limits:}
\begin{enumerate}  \romanenum
\item \label{i:ex:sinfty} The \Fraisse limit of all finite sets without structure is a countably infinite set. The corresponding group is $S_\infty$, the group of all permutations of this set.
\item \label{i:ex:Q} The \Fraisse limit of all finite linear orders is the countable dense linear order without endpoints $(\Q, <)$. We denote the corresponding automorphism group by $\Aut(\Q)$.
\item \label{i:ex:Cantor} The \Fraisse limit of all finite Boolean algebras is the countable atomless Boolean algebra which is isomorphic to the algebra of all clopen subsets of the Cantor space $2^\N$. The corresponding automorphism group is $\Homeo(2^\N)$, the group of all homeomorphisms of $2^\N$.
\item \label{i:ex:GL} The \Fraisse limit of all finite vector spaces over a fixed finite field $\F_q$ is the infinite-dimensional vector space over $\F_q$. The automorphism group is the general linear group $\GL(\infty, \F_q)$.
\item \label{i:ex:R} The \Fraisse limit of all finite graphs is the \df{random graph}, the unique countable graph $\bR$ such that for every two finite disjoint sets of vertices $U, V$, there exists a vertex $x$ which is connected by an edge to all vertices in $U$ and to no vertices in $V$. We denote its automorphism group by $\Aut(\bR)$.

\setcounter{ExamplesEnum}{\value{enumi}}
\end{enumerate}
There are also many other $\omega$-categorical structures, including, for example, certain groups \cite{Hodges1993}, and a variety of combinatorial structures \cite{Kechris2005}. There is a rather extensive literature devoted to the subject; we refer the interested reader to the volume \cite{Kaye1994}, or the more recent survey Macpherson~\cite{Macpherson2010p} and the references therein. We also indicate some ways to construct new oligomorphic groups from old ones in Proposition~\ref{p:closure}.

The main theorem describing the unitary representations of oligomorphic groups is the following.
\begin{theorem} \label{th:int:main}
Let $G$ be an oligomorphic group. Then every irreducible unitary representation of $G$ is of the form $\Ind_{C(V)}^G(\sigma)$, where $V \leq G$ is an open subgroup, $C(V)$ is the commensurator of $V$, $V \unlhd C(V)$, and $\sigma$ is an irreducible representation of the \emph{finite} group $C(V)/V$. Moreover, every unitary representation of $G$ is a sum of irreducibles.
\end{theorem}
We also provide a criterion when two irreducible representations as above are isomorphic (Proposition~\ref{p:ind-reps}).

As every oligomorphic group has only countably many distinct open subgroups (Corollary~\ref{c:ctbl-open}), the theorem shows that every oligomorphic group has only countably many irreducible representations.

Our methods are quite different from the approach of Olshanski. In particular, his semigroup method only applies when the group is obtained as the completion of an inductive limit of subgroups, which is not the case for many oligomorphic groups (for example, $\Aut(\Q)$). On the other hand, we have borrowed an important idea from Lieberman: the use of a weak limit point in the proof of Theorem~\ref{th:int:main}.

If one is given a realization of a closed oligomorphic group as the automorphism group of a countable structure, it is possible to give a more concrete description of its representations in terms of the structure. For example, for the automorphism group of the random graph, all irreducible representations can be obtained in the following way. One takes a finite (induced) subgraph $\bA \sub \bR$ and sets $V$ to be the pointwise stabilizer of $\bA$. Then $C(V)$ is the setwise stabilizer of $\bA$ and $C(V)/V \cong \Aut(\bA)$. So in this case, the irreducible representations of $\Aut(\bR)$ are obtained by induction from irreducible representations of automorphism groups of finite graphs (and in fact, this correspondence is one-to-one if one makes the obvious identifications). See Section~\ref{s:examples} for more details.

As a corollary of the classification of the representations of $\Aut(\Q)$, we obtain that the group $\Homeo^+(\R)$ has no non-trivial unitary representations (this is a special case of a result of Megrelishvili), cf. Corollary~\ref{c:HomeoR}.

As a further application, we establish property (T) for a large class of oligomorphic groups. Our technique is quite similar to the one used by Bekka~\cite{Bekka2003} to prove that the unitary group has property (T) (for which he used Kirillov and Olshanski's classification of its representations).
\begin{theorem}
All of the examples \eqref{i:ex:sinfty}--\eqref{i:ex:R} above have property (T).
\end{theorem}
In all of those groups, it is also possible to find explicit finite Kazhdan sets. We also have a more general result (Theorem~\ref{th:propT-gen}), which requires some additional terminology to state. The question whether all closed oligomorphic groups have property (T) remains open.

The paper is organized as follows. In Section~\ref{s:olig-groups}, we recall the definition of Roelcke precompactness and provide a model-theoretic characterization of Roelcke precompact subgroups of $S_\infty$. In Section~\ref{s:GelfRaikov}, we prove some basic results about representations of non-archimedean groups, including Theorem~\ref{th:GelfRaikov}. In Section~\ref{s:rep-olig}, we prove the main theorem. Section~\ref{s:examples} is devoted to some model-theoretic considerations and calculations in specific examples. In Section~\ref{s:propT}, we discuss property (T).

\subsection*{Notation} If $G$ is a group and $g, x \in G$, $g^x$ denotes the conjugate $xgx^{-1}$. Note that this is the conjugation action on the left, so that $(g^x)^y = g^{yx}$. If $G$ is the automorphism group of a structure $\bX$ and $\bar a \in \bX^n$ is a tuple, $G_{\bar a}$ denotes the stabilizer of all elements of $\bar a$. If $\bA \sub \bX$ is a substructure, $G_\bA$ is the \df{setwise stabilizer} of $\bA$ (the set of all $g$ such that $g \cdot \bA = \bA$) and $G_{(\bA)}$ is the \df{pointwise stabilizer} (the set of all $g$ such that $g \cdot a = a$ for all $a \in \bA$). If $X$ is a set, $X^{[n]}$ denotes the set of all subsets of $X$ of size $n$.

A \df{representation} of a group $G$ is always a unitary representation. If $\pi$ is a representation, $\mcH(\pi)$ denotes its Hilbert space. All Hilbert spaces are complex.

\subsection*{Acknowledgements} I would like to thank Ita\"\i\ Ben Yaacov and C.~Ward Henson for an important insight for the proof of Lemma~\ref{l:left-right}, Martin Hils for explaining to me some basic model theory, and H.~Dugald Macpherson for pointing out an error in a preliminary version of the paper. I am grateful to the Fields Institute in Toronto for its hospitality while this paper was being finished. I am also grateful to the referee for a number of helpful comments.

\section{Oligomorphic groups and Roelcke precompactness} \label{s:olig-groups}
\subsection{Roelcke precompact topological groups}
A topological group $G$ is \df{precompact} iff the completion of its left uniformity is compact iff the completion of its right uniformity is compact. In that case, the common completion of the left and the right uniformities has the structure of a compact group in which $G$ embeds topologically as a dense subgroup. In particular, if $G$ is Polish, $G$ is precompact iff it is compact. The condition that the left uniformity on a group $G$ is precompact can be written as follows: for every neighborhood $U$ of the identity, there exists a finite set $F \sub G$ such that $FU = G$.

There exists a weaker notion of precompactness that will be central for this paper. A topological group is called \df{Roelcke precompact} if the completion of its \df{lower uniformity} (the greatest lower bound of the left and right uniformities) is compact. We will find it, however, more convenient to work with the following direct definition.
\begin{defn} \label{df:Roelcke}
A topological group $G$ is \df{Roelcke precompact} if for every neighborhood of the identity $U$, there exists a finite set $F \sub G$ such that $G = UFU$.
\end{defn}
The notion of Roelcke precompactness is weak enough to include many interesting non-compact examples but still sufficiently powerful to allow the generalisation of some results valid for compact groups. Roelcke precompact groups were introduced by Roelcke--Dierolf~\cite{Roelcke1981}, who also gave the first examples, and were later studied by many authors, including \Uspenskii and Glasner. Known examples of Roelcke precompact groups include the unitary group of a separable, infinite-dimensional Hilbert space (\Uspenskii~\cite{Uspenskii1998}), the isometry group of the bounded Urysohn metric space (\Uspenskii~\cite{Uspenskii2008}), and the automorphism group of a standard probability space (Glasner~\cite{Glasner2009p}). We also note that Roelcke precompactness is only interesting for ``infinite-dimensional'' groups: a locally compact group is Roelcke precompact iff it is compact (to see this, note that if $U$ is a compact neighborhood of the identity, then $G = UFU$ is compact). We suggest the survey \cite{Uspenskii2002} by \Uspenskii as a general reference.

We start by showing that the class of Roelcke precompact groups enjoys most of the closure properties of the class of compact groups with the notable exception of closure under taking closed subgroups. (In fact, as was shown in \cite{Uspenskii2008}, every Polish group embeds as a closed subgroup of a Polish Roelcke precompact group.)
\begin{prop} \label{p:closure}
The following statements hold:
\begin{enumerate} \romanenum
\item \label{i:pcl:quot} If $\pi \colon G \to H$ is a continuous homomorphism with a dense image and $G$ is Roelcke precompact, then so is $H$.
\item \label{i:pcl:prod} If $G_1$ and $G_2$ are Roelcke precompact, then so is $G_1 \times G_2$.
\item \label{i:pcl:inv} The inverse limit of an inverse system of Roelcke precompact groups is Roelcke precompact. In particular, an arbitrary product of Roelcke precompact groups is Roelcke precompact.
\item \label{i:pcl:open} If $G$ is Roelcke precompact and $H \leq G$ is open, then $H$ is Roelcke precompact.
\item \label{i:pcl:ext} If $N$ is a normal subgroup of $G$ such that both $N$ and $G/N$ are Roelcke precompact, then so is $G$.
\end{enumerate}
\end{prop}
\begin{proof}
\eqref{i:pcl:quot}. Let $U \sub H$ be an open neighborhood of $1$ . Find an open $V$ such that $1 \in V^2 \sub U$. Let $F \sub G$ be finite such that $G = \pi^{-1}(V)F\pi^{-1}(V)$. We check that $U\pi(F)U = H$. Let $h \in H$. By the density of $\pi(G)$ in $H$, there exists $g \in G$ such that $\pi(g) \in hV^{-1}$. Then $g = v_1 f v_2$ for some $v_1, v_2 \in \pi^{-1}(V)$ and $f \in F$. Finally, $h \in \pi(g)V = \pi(v_1) \pi(f) \pi(v_2) V \sub U\pi(F)U$.

\eqref{i:pcl:prod}. If $U_1 \times U_2$ is an open neighborhood of the identity in $G_1 \times G_2$ and $F_1 \sub G_1$, $F_2 \sub G_2$ are finite such that $U_1F_1U_1 = G_1$, $U_2F_2U_2 = G_2$, then $G_1 \times G_2 = (U_1 \times U_2)(F_1 \times F_2)(U_1 \times U_2)$.

\eqref{i:pcl:inv}. Let $G = \varprojlim H_i$. Let $U = \pi_i^{-1}(U_i)$ be an open neighborhood of $1$ in $G$, where $U_i$ is an open neighborhood of $1$ in $H_i$. Then there exists a finite $F \sub H_i$ with $H_i = U_iFU_i$. If $F' \sub G$ is finite with $\pi_i(F') = F$, we check that $UF'U^2 = G$. Indeed, let $x$ in $G$. Then there exist $u_1, u_2 \in G$, $f \in F'$ such that $\pi(x) = \pi(u_1 f u_2)$. Then there is $h \in \ker \pi_i$ such that $x = u_1 f u_2 h \in UF'U^2$ (because $\ker \pi_i \sub U$ by the definition of $U$).

\eqref{i:pcl:open}. Let $U \sub H$ be an open neighborhood of $1$. Then $U$ is open in $G$ and there exists a finite $F \sub G$ such that $G = UFU$. Now one easily checks that $H = U(F \cap H)U$.

\eqref{i:pcl:ext}. Let $U \sub G$ be an open neighborhood of $1$. Let $Q = G/N$ and denote by $\pi \colon G \to Q$ the quotient map. Then $\pi(U)$ is open in $Q$. Find a finite $F \sub G$ such that $\pi(U)\pi(F)\pi(U) = Q$ and assume also that $1 \in F$. Let $V = N \cap \bigcap_{f \in F} U^f$ and note that $V$ is relatively open in $N$. Find a finite $B \sub N$ such that $VBV = N$. We claim that $U^2BFU^2 = G$. Indeed, fix $x \in G$. Find $f \in F$ and $u_1, u_2 \in U$ such that $\pi(u_1 f u_2) = \pi(x)$. Then there exists $h \in N$ with $x = h u_1 f u_2$. Find $v_1, v_2 \in V$ and $b \in B$ such that $u_1^{-1} h u_1 = v_1 b v_2$. Finally, we have
\[ \begin{split}
x &= h u_1 f u_2 = u_1 (u_1^{-1} h u_1) f u_2 = u_1 v_1 b v_2 f u_2 \\
  &= u_1 v_1 b f (f^{-1} v_2 f) u_2 \in U^2 BF U^2,
\end{split} \]
finishing the proof.
\end{proof}

We end this subsection with a simple application. In \cite{Rosendal2009}, Rosendal introduced property (OB), intended as a generalization of the well known properties (FA) and (FH). A topological group has \df{property (OB)} if every time it acts (separately) continuously by isometries on a metric space, every orbit is bounded.
\begin{prop} \label{p:OB}
Every Roelcke precompact group has property (OB).
\end{prop}
\begin{proof}
Let $G$ be a Roelcke precompact group and $G \actson (X, d)$ a separately continuous action on a metric space by isometries. For every point $x_0 \in X$, the function $g \mapsto d(x_0, g \cdot x_0)$ is uniformly continuous in the lower uniformity, as can be seen from the inequality
\[ \begin{split}
 d(x_0, h_1 g h_2 \cdot x_0) &= d(h_1^{-1} \cdot x_0, g h_2 \cdot x_0) \\
 &\leq d(h_1^{-1} \cdot x_0, x_0) + d(x_0, g \cdot x_0) + d(g \cdot x_0, gh_2 \cdot x_0) \\
 &= d(h_1^{-1} \cdot x_0, x_0) + d(x_0, g \cdot x_0) + d(x_0, h_2 \cdot x_0).
\end{split}
\]
It thus extends to the compact completion and must be bounded.
\end{proof}

\subsection{Permutation groups and closed subgroups of $S_\infty$}
We now concentrate on the main objects of study in this paper, namely infinite permutation groups. Let $S_\infty$ be the group of all permutations of a countable infinite set $\bX$. It becomes naturally a topological group if equipped with the pointwise convergence topology, where $\bX$ is taken to be discrete. A \df{permutation group} is a topological subgroup of the group of all permutations of $\bX$.

It is well known that the topological groups that can be realized as permutation groups are exactly the separable topological groups that admit a countable basis at the identity consisting of open subgroups (those groups are often called \df{non-archimedean}). The basis of open subgroups is given by the pointwise stabilizers of finite subsets of $\bX$.

A natural way in which closed permutation groups arise in practice is as automorphism groups of countable structures in model theory, for example, automorphism groups of countable graphs, countable orders, or various algebraic structures. Of special interest to us will be the oligomorphic groups (see Definition~\ref{df:olig}) for which it is possible to translate back and forth between model-theoretic and permutation group-theoretic language. For a gentle introduction to the subject of oligomorphic groups, we refer the reader to Cameron~\cite{Cameron1990}. It turns out that there is a close connection between the properties of being Roelcke precompact and oligomorphic. To see this, we first reformulate Definition~\ref{df:Roelcke} for non-archimedean groups: a topological subgroup $G$ of $S_\infty$ is Roelcke precompact iff for every open subgroup $V \leq G$, the set of double cosets
\[
V \backslash G / V = \set{VxV : x \in G}
\]
is finite.

\begin{theorem} \label{th:char-olig}
For a topological subgroup $G \leq S_\infty$, the following are equivalent:
\begin{enumerate} \romanenum
\item \label{i:tco:RC} $G$ is Roelcke precompact;
\item \label{i:tco:olig} for every continuous action $G \actson \bX$ on a countable, discrete set $\bX$ with finitely many orbits, the induced action $G \actson \bX^n$ has finitely many orbits for each $n$;
\item \label{i:tco:inv-limit} $G$ can be written as an inverse limit of oligomorphic groups.
\end{enumerate}
\end{theorem}
\begin{proof}

\eqref{i:tco:RC} $\Implies$ \eqref{i:tco:olig}. Without loss of generality, we can suppose that the action $G \actson \bX$ is transitive. We use induction on $n$. The case $n = 1$ is given by the hypothesis. For the induction step $n \to n+1$, it suffices to find for every $\bar a \in \bX^n$, a finite set $B(\bar a) \sub \bX$ such that for every $d \in \bX$, there is $h \in G_{\bar a}$ and $b \in B(\bar a)$ such that $d = h \cdot b$. Then if $\set{\bar a_1, \ldots, \bar a_s}$ is a complete set of representatives for the orbits of $G \actson \bX^n$, $\set{(\bar a_i, b) : i \leq s, b \in B(\bar a_i)}$ will be a complete set of representatives for the action $G \actson \bX^{n+1}$. Indeed, let $(\bar c, d) \in \bX^{n+1}$. Using the induction hypothesis, find $g \in G$ such that $g \cdot \bar a_i = \bar c$. Find $h \in G_{\bar a_i}$ and $b \in B(\bar a_i)$ such that $g^{-1} \cdot d = h \cdot b$. Then one has
\[
gh \cdot (\bar a_i, b) = g \cdot (a_i, h \cdot b) = (\bar c, d).
\]

Fix now $\bar a \in \bX^n$, let $c_0$ be an arbitrary element of $\bX$, and let $\set{G_{\bar a c_0} g_0 G_{\bar a c_0}, \ldots, \linebreak G_{\bar a c_0} g_k G_{\bar a c_0}}$ be a complete list of the double cosets of $G_{\bar a c_0}$. Set $B(\bar a) = \set{g_i \cdot c_0 : i = 1, \ldots, k}$. Let now $d \in \bX$ be arbitrary, $d = g \cdot c_0$ (using the transitivity of the action). Let $i, h_1, h_2$ be such that $h_1, h_2 \in G_{\bar a c_0}$ and $g = h_1 g_i h_2$. We have
\[
d = g \cdot c_0 = h_1 g_i h_2 \cdot c_0  = h_1 g_i \cdot c_0,
\]
finishing the proof.

\eqref{i:tco:olig} $\Implies$ \eqref{i:tco:inv-limit}. Let $\{V_n : n \in \N\}$ be a basis at $1_G$ of open subgroups such that for all $n$, $V_{n+1} \leq V_n$. Then for each $n$, $G$ acts continuously by permutations on the discrete set $G/V_n$ and this gives rise to a continuous homomorphism $\pi_n \colon G \to \mathrm{Sym}(G/V_n)$. The groups $\{\pi_n(G) : n \in \N \}$ form a directed system (there are natural maps $\pi_{n+1}(G) \to \pi_n(G)$) and its inverse limit is isomorphic to $G$. Each $\pi_n(G)$ is oligomorphic by the hypothesis.

\eqref{i:tco:inv-limit} $\Implies$ \eqref{i:tco:RC}. In view of Proposition~\ref{p:closure}, it suffices to show that if $G$ is oligomorphic, then it is Roelcke precompact. Let $G$ be a group of permutations of $\bX$ such that the action $G \actson \bX$ is oligomorphic. It suffices to show that every stabilizer $G_{\bar a}$ for $\bar a \in \bX^n$ has finitely many double cosets. If $\set{(\bar a, g_1 \cdot \bar a), \ldots, (\bar a, g_k \cdot \bar a)}$ is a complete list of representatives for the $G$-orbits on $\bX^{2n}$ that are subsets of $G \cdot \bar a \times G \cdot \bar a$, then it is easy to check that $\set{G_{\bar a}g_iG_{\bar a} : i = 1, \ldots, k}$ exhausts all double cosets of $G_{\bar a}$.
\end{proof}
%

\begin{remark*}
A proof that every (even approximately) oligomorphic group is Roelcke precompact is essentially contained in \cite{Rosendal2009}. Also, some special cases of the above theorem had been known before: \Uspenskii had shown that $S_\infty$ and $\Homeo(2^\N)$ are Roelcke precompact \cites{Uspenskii2001,Uspenskii2002}.
\end{remark*}

A basic application of the theorem is the following corollary which had been noted before by many authors.
\begin{cor} \label{c:ctbl-open}
Every oligomorphic group has only countably many distinct open subgroups.
\end{cor}
\begin{proof}
Fix a countable basis at the identity of open subgroups $\set{V_i}$. Every open subgroup contains a basic open subgroup and is therefore a union of finitely many double cosets of some $V_i$.
\end{proof}


Now we turn to some basic group-theoretic lemmas about oligomorphic groups that will be used later. Recall that two subgroups of a group are \df{commensurate} if their intersection has finite index in both. If $H \leq G$ is a subgroup, define the \df{commensurator of $H$ in $G$} to be
\[
\Comm_G(H) = \set{g \in G : H \text{ and } H^g \text{ are commensurate}}.
\]
It is a standard fact that $\Comm_G(H)$ is a subgroup of $G$ containing $H$. If $H_1$ and $H_2 \leq G$ are commensurate, then $\Comm_G(H_1) = \Comm_G(H_2)$. Note also that the number of left cosets in $HgH$ is equal to $[H : H \cap H^g]$ and the number of right cosets is $[H^g : H \cap H^g]$. The following lemma will be particularly useful for studying commensurators in oligomorphic groups. I am grateful for the idea of the proof to Ita\"\i~Ben~Yaacov and C.~Ward~Henson.
\begin{lemma}\label{l:left-right}
Let $G$ be a Roelcke precompact group and $V \leq G$ an open subgroup. Then, for every $x \in G$, the double coset $VxV$ contains finitely many left cosets of $V$ iff it contains finitely many right cosets of $V$.
\end{lemma}
\begin{proof}
Note first that if $H_1, H_2, H_3 \leq G$ and $H_1H_2$ contains finitely many $H_2$-cosets and $H_2H_3$ contains finitely many $H_3$-cosets, then $H_1H_3 \sub H_1H_2H_3$ also contains only finitely many $H_3$-cosets.

For $H \leq G$, denote
\[ \begin{split}
\mcF(H) &= \set{yH : HyH \text{ contains finitely many left } H\text{-cosets}} \\
  &= \set{yH : HH^y \text{ contains finitely many } H^y\text{-cosets}}.
\end{split} \]
Note that $\mcF(H)$ is exactly the union of double cosets that contain only finitely many left cosets and as $G$ has only finitely many double cosets of $V$, $\mcF(V)$ is finite. Now we show that
\begin{equation} \label{eq:Ftrans}
xV, yV \in \mcF(V) \implies xyV \in \mcF(V).
\end{equation}
Indeed, we have that $VV^x$ contains finitely many $V^x$ cosets and $VV^y$ contains finitely many $V^y$ cosets. Conjugating the latter by $x$, we obtain that $V^xV^{xy}$ contains finitely many $V^{xy}$ cosets. Now applying the observation in the beginning of the proof, we have that $VV^{xy}$ contains finitely many $V^{xy}$ cosets, i.e. $xyV \in \mcF(V)$.

Suppose now that $x \in G$ is such that $xV \in \mcF(V)$ and consider the map $\Phi \colon G/V \to G/V$ defined by $\Phi(yV) = xyV$. By \eqref{eq:Ftrans}, $\Phi(\mcF(V)) \sub \mcF(V)$ and as $\mcF(V)$ is finite and $\Phi$ is injective, in fact, $\Phi(\mcF(V)) = \mcF(V)$. Therefore there exists $yV \in \mcF(V)$ such that $\Phi(yV) = V$, i.e. $yV = x^{-1}V$. So we conclude that $x^{-1}V \in \mcF(V)$ or equivalently, that $VxV$ contains finitely many right cosets. The other direction of the statement is obtained by replacing $x$ with $x^{-1}$.
\end{proof}

As a corollary, we obtain that we can define the commensurator of an open subgroup in a Roelcke precompact group as
\begin{equation} \begin{split} \label{eq:defcomm}
\Comm_G(V) &= \bigcup \set{VgV : VgV \text{ contains finitely many left cosets of } V} \\
    &= \bigcup \set{VgV : VgV \text{ contains finitely many right cosets of } V}.
\end{split} \end{equation}

\begin{lemma} \label{l:commen}
Let $G$ be Roelcke precompact and $V \leq G$ be open. Then the following hold:
\begin{enumerate} \romanenum
\item \label{i:lcm:fi} $[\Comm_G(V) : V] < \infty$;
\item \label{i:lcm:idemp} $\Comm_G(\Comm_G(V)) = \Comm_G(V)$;
\item \label{i:lcm:2cms} If $V_1, V_2 \leq G$ are open and $\Comm_G(V_1)$ and $\Comm_G(V_2)$ are commensurate, then $\Comm_G(V_1) = \Comm_G(V_2)$.
\end{enumerate}
\end{lemma}
\begin{proof}
\eqref{i:lcm:fi}. Since $V \leq \Comm_G(V)$, $V$ is open, and $G$ is Roelcke precompact, $\Comm_G(V)$ is a union of finitely many double cosets of $V$. By \eqref{eq:defcomm}, every double coset $VxV$ in $\Comm_G(V)$ contains only finitely many left cosets, so the claim follows.

\eqref{i:lcm:idemp}. Follows from \eqref{i:lcm:fi}.


\eqref{i:lcm:2cms}. We have
\[
\Comm_G(V_1) = \Comm_G(\Comm_G(V_1)) = \Comm_G(\Comm_G(V_2)) = \Comm_G(V_2).
\]
\end{proof}
We see that $\Comm_G(V)$ is the maximal subgroup of $G$ containing $V$ in which $V$ is of finite index. In view of \eqref{i:lcm:idemp}, we say that an open subgroup $H$ of a Roelcke precompact, non-archimedean group $G$ is \df{a commensurator} if $\Comm_G(H) = H$.

\subsection{Examples from model theory} \label{ss:examples-mth}
A natural class of permutation groups is obtained from model theory. A \df{signature} (or a \df{language}) $\mcL$ is a collection of relation and function symbols, where each symbol has a certain \df{arity} $n(\cdot)$. A structure for $\mcL$ is a set $\bX$ together with interpretations for the symbols: for each relation symbol $R$, a relation $R^\bX \sub \bX^{n(R)}$ and for each function symbol $F$, a function $F^\bX \colon \bX^{n(F)} \to \bX$. An \df{automorphism} of the structure $\bX$ is a permutation that preserves the relations and the functions. A structure is \df{relational} if the signature has no functions symbols.

Every closed permutation group can be obtained as the automorphism group of a relational structure: one just adds a relation for every orbit on $\bX^n$ for all $n$ (see, for example, \cite{Becker1996} for more details). An important model-theoretic characterization of the structures whose automorphism groups are oligomorphic is given by the following classical theorem (see \cite{Hodges1993} for a proof).

\begin{theorem}[Ryll-Nardzewski, Engeler, Svenonius]
For a countable structure $\bX$, the following are equivalent:
\begin{itemize}
\item $\bX$ is $\omega$-categorical;
\item $\Aut(\bX) \actson \bX$ is an oligomorphic permutation group.
\end{itemize}
\end{theorem}

A structure $\bX$ is called \df{$\omega$-categorical} if $\bX$ is the unique (up to isomorphism) countable model of the first-order theory of $\bX$.

An especially attractive situation is when $\bX$ is \df{homogeneous} in the following strong sense: every isomorphism between finite substructures of $\bX$ extends to a full automorphism of $\bX$ (sometimes those structures are called \df{ultrahomogeneous}). A classical theorem of \Fraisse describes the homogeneous structures as the \Fraisse limits of classes of finite structures satisfying a certain amalgamation property \cite{Hodges1993}*{Section~7.4}. A homogeneous structure $\bX$ is $\omega$-categorical iff for every $n$ there are only finitely many isomorphism types of substructures of $\bX$ generated by $n$ elements. In particular, every homogeneous structure in a finite, relational signature is $\omega$-categorical. We also see that all examples given in the introduction are $\omega$-categorical. We add one slightly less known, but instructive, example to the list.

\smallskip \noindent \textbf{Examples (cont.):}
\begin{enumerate} \romanenum
\setcounter{enumi}{\value{ExamplesEnum}}
\item \label{i:ex:Cherlin} (Cherlin and Hrushovski) Consider a signature with infinitely many relation symbols $\set{E_n}_{n \geq 1}$, where $E_n$ is of arity $2n$. Let $\mcX$ be the class of all finite structures in this signature, where each $E_n$ is interpreted as an equivalence relation on \emph{subsets} of the structure of size $n$ with at most $2$ equivalence classes. For every $k \in \N$, there are only finitely many structures in $\mcX$ of size $k$, so the \Fraisse limit $\bX$ of $\mcX$ is $\omega$-categorical. What is remarkable about this structure is that its automorphism group $G$ has a quotient isomorphic to $(\Z / 2\Z)^\N$. The reason for this is that $G$ acts on the sets $\bX^{[n]} / E_n$ all of which have size $2$. Extending this example, Evans and Hewitt~\cite{Evans1990} constructed for every profinite, metrizable group $H$ an oligomorphic group $G$ such that $H$ is a quotient of $G$.

\end{enumerate}

Typical homogeneous structures that are often encountered in the literature which are \emph{not} $\omega$-categorical are the discrete structures that arise as approximations of continuous ones: the rational Urysohn metric space, countable measured Boolean algebras, etc.

\section{Representations of permutation groups} \label{s:GelfRaikov}
A \df{unitary representation} of a topological group $G$ is a \df{strongly continuous} homomorphism $\pi \colon G \to U(\mcH)$ to the unitary group of some Hilbert space $\mcH$ (i.e. the map $G \to \mcH$, $g \mapsto \pi(g)\xi$ is continuous for every $\xi \in \mcH$). A permutation group $G \actson \bX$ has some natural representations defined using the action on $\bX$, namely, $G \actson \ell^2(\bX^n)$ for $n \in \N$. Those representations clearly separate the points of $G$. In this section, we show that \emph{irreducible} representations also separate points. (A representation $\pi$ is \df{irreducible} if $\mcH(\pi)$ does not have non-trivial subspaces invariant under $\pi$.) This can be considered as a version of the classical Gelfand--Raikov~\cite{Gelfand1943} theorem for subgroups of $S_\infty$.

If $\pi$ is a representation of $G$ and $V \leq G$, denote by $\mcH_V(\pi)$ the closed subspace of fixed points of $V$. We start with a simple but key lemma a version of which had been previously used by Lieberman~\cite{Lieberman1972} and Glasner--Weiss~\cite{Glasner2005a}.
\begin{lemma} \label{l:dense-fp}
Let $G$ be a subgroup of $S_\infty$ and $\set{V_i : i \in \N}$ be a basis at the identity for $G$ consisting of open subgroups. Then for any continuous unitary representation $\pi$ of $G$, the space $\bigcup_i \mcH_{V_i}(\pi)$ is dense in $\mcH(\pi)$.
\end{lemma}
\begin{proof}
Let $\xi_0 \in \mcH(\pi)$ be an arbitrary vector and fix $\eps > 0$. As the representation is continuous, there is $i \in \N$ such that $\pi(V_i) \xi_0$ is contained in the ball with radius $\eps$ around $\xi_0$. Let $C$ be the closure of the convex hull of $\pi(V_i) \xi_0$ and let $\eta$ be the unique element of least norm in $C$. As $\pi(V_i)C = C$ and $\pi$ preserves the norm, $\eta$ is a fixed point of $\pi(V_i)$. Also, by the choice of $V_i$, $\nm{\eta - \xi_0} \leq \eps$.
\end{proof}

Recall that a continuous function $\phi \colon G \to \C$ is called \df{positive definite} if for every $x_1, \ldots, x_n \in G$ and $c_1, \ldots, c_n \in \C$,
\begin{equation} \label{eq:posdef}
\sum_{i, j = 1}^n \phi(x_j^{-1}x_i)c_i \conj{c_j} \geq 0,
\end{equation}
i.e. the matrix $\big(\phi(x_j^{-1}x_i)\big)_{i, j}$ is positive-definite. If $\pi$ is a representation of $G$ and $\xi \in \mcH(\pi)$, the function $x \mapsto \ip{\pi(x)\xi, \xi}$ is positive definite and conversely, the GNS construction produces from a positive definite function $\phi$ a representation $\pi$ and a \df{cyclic vector} $\xi \in \mcH(\pi)$ (i.e. such that the linear span of the orbit $\pi(G)\xi$ is dense in $\mcH(\pi)$) such that
\[
\phi(x) = \ip{\pi(x)\xi, \xi} \quad \text{for all } x \in G
\]
(see \cite{Bekka2008}*{Appendix C} for more details). In particular, $|\phi(x)| \leq \phi(1)$ for all $x \in G$. We now have the following basic observation.
\begin{lemma} \label{l:const-dblcst}
Let $G$ be a group and $H \leq G$ a subgroup. If $\phi$ is a positive definite function that is constant on $H$, then it is constant on double cosets of $H$.
\end{lemma}
\begin{proof}
Let $\pi$ and $\xi \in \mcH(\pi)$ be such that $\phi(x) = \ip{\pi(x)\xi, \xi}$ for $x \in G$. Then by the hypothesis, for any $h \in H$, $\ip{\pi(h)\xi, \xi} = \ip{\pi(1)\xi, \xi} = \nm{\xi}^2$, i.e. $\xi \in \mcH_H(\pi)$. Now we have for any $x \in G$ and $h_1, h_2 \in H$:
\[
\phi(h_1 x h_2) = \ip{\pi(h_1 x h_2)\xi, \xi} = \ip{\pi(x)\pi(h_2)\xi, \pi(h_1^{-1})\xi}
    = \ip{\pi(x)\xi, \xi} = \phi(x),
\]
finishing the proof.
\end{proof}

Let
\[
\mcP_1(G) = \set{\phi \colon G \to \C : \phi \text{ is positive definite and } \phi(1) = 1}
\]
and if $V \leq G$ is open, let also
\[
\mcP_V(G) = \set{\phi \in \mcP_1(G) : \phi(v) = 1 \text{ for all } v \in V}.
\]
By Lemma~\ref{l:const-dblcst}, we can consider $\mcP_V(G)$ as a subset of $\ell^\infty(V \backslash G / V)$. $\mcP_V(G)$ is convex and bounded and by the definition \eqref{eq:posdef}, it is also closed in the weak$^*$ topology of $\ell^\infty(V \backslash G / V)$ and thus compact.
\begin{lemma} \label{l:extr-points}
If $\phi$ is an extreme point of $\mcP_V(G)$, then it is also an extreme point of $\mcP_1(G)$.
\end{lemma}
\begin{proof}
Suppose that $\phi \in \mcP_V(G)$ and $\psi_1, \psi_2 \in \mcP_1(G)$ and $t \in (0, 1)$ are such that $\phi = t \psi_1 + (1 - t)\psi_2$. For every $v \in V$, we have
\[
1 = \phi(v) = t \Re \psi_1(v) + (1 - t) \Re \psi_2(v) \leq t \psi_1(1) + (1 - t) \psi_2(1) = 1,
\]
showing that we must have equality in the middle, i.e. $\Re \psi_1(v) = \Re \psi_2(v) = 1$ for all $v \in V$. As $|\psi_1(v)|, |\psi_2(v)| \leq 1$, this implies that $\psi_1(v) = \psi_2(v) = 1$. Thus $\psi_1, \psi_2 \in \mcP_V(G)$, proving the lemma.
\end{proof}

We now see that the classical proof of Gelfand--Raikov extends to our situation.
\begin{proof}[Proof of Theorem~\ref{th:GelfRaikov}]
It suffices to show that for every $x \in G$, $x \neq 1$, there is an irreducible representation $\pi$ such that $\pi(x) \neq \pi(1)$. Recall that a representation $\pi$ with a cyclic unit vector $\xi$ is irreducible iff the corresponding positive definite function is an extreme point of $\mcP_1(G)$ \cite{Bekka2008}*{Theorem~C.5.2}. Let now $1 \neq x \in G$. Let $V \leq G$ be an open subgroup such that $x \notin V$. Consider the positive definite function $\chi_V$ (the characteristic function of $V$) which corresponds to the representation $G \actson \ell^2(G/V)$ with cyclic vector $\delta_V$. We have that $\chi_V \in \mcP_V(G)$ and $\chi_V(x) \neq \chi_V(1)$. Consider now the weak$^*$ closed, convex set $C = \set{\phi \in \mcP_V(G) : \phi(x) = \phi(1)}$.
As $C \subsetneq \mcP_V(G)$, by the Krein--Milman theorem, there exists an extreme point $\phi$ of $\mcP_V(G)$ such that $\phi \notin C$. By Lemma~\ref{l:extr-points}, $\phi$ is also an extreme point of $\mcP_1(G)$, producing the required irreducible representation.
\end{proof}

\section{Representations of oligomorphic groups} \label{s:rep-olig}
Let $G$ be a subgroup of $S_\infty$ and $G \actson Y$ be a continuous action on a discrete, countable set $Y$. There is a natural associated representation of $G$ on $\ell^2(Y)$ and if $Y = \bigsqcup_i Y_i$ is the decomposition of $Y$ into orbits, we have that $\ell^2(Y) = \boplus_i \ell^2(Y_i)$. Therefore we can as well suppose that the action $G \actson Y$ is transitive; in this case, the corresponding representation is just the \df{quasi-regular representation} $\ell^2(G/V)$ for some open subgroup $V$ of $G$. In order to describe those, we recall the notion of induced representation.

Let $G$ be a topological group and $H$ be an \df{open} subgroup of $G$. Let $\sigma$ be a representation of $H$. The \df{induced representation} $\Ind_H^G(\sigma)$ is defined as follows. Let $T$ be a complete system of left coset representatives of $H$ in $G$. Let $M$ be the space of all functions $f \colon G \to \mcH(\sigma)$ for which
\begin{equation} \label{eq:induced-cond}
 f(gh) = \sigma(h^{-1})f(g) \quad \text{for all } g \in G, h \in H.
\end{equation}
In particular, for $f \in M$, $\nm{f(x)}$ is constant on left cosets of $H$. For $f \in M$, define
\begin{equation} \label{eq:induced-nm}
\nm{f} = \Big(\sum_{g \in T} \nm{f(g)}^2 \Big)^{1/2}
\end{equation}
and note that because of the above observation, $\nm{f}$ does not depend on the choice of $T$. Let $\mcH = \set{f \in M : \nm{f} < \infty}$. Then the representation $\Ind_H^G(\sigma)$ on the Hilbert space $\mcH$ is defined by
\[
\big(\Ind_H^G(\sigma)(g) \cdot f \big)(x) = f(g^{-1}x).
\]
As $H$ is open, the representation $\Ind_H^G(\sigma)$ is continuous. For example, the quasi-regular representation $\ell^2(G/H)$ can be written as $\Ind_H^G(1_H)$, where $1_H$ is the trivial one-dimensional representation of $H$.

We note that as we only need to induce from open subgroups $H \leq G$, the homogeneous space $G/H$ always carries the counting measure and we are spared the measure-theoretic complications that occur in the locally compact setting. For more details on induced representations, see, for example, \cite{Bekka2008}*{Appendix~E}.

Suppose now that $G$ is Roelcke precompact and fix an open subgroup $V \leq G$. Let $H$ be a subgroup of $G$ such that $V \unlhd H$. As for normal subgroups double cosets coincide with left cosets, $V$ has finite index in $H$. Denote by $K$ the finite group $H/V$ and by $\lambda_K$ the left-regular representation of $K$, which we will also consider as a representation of $H$. Then using the theorem about induction in stages (\cite{Bekka2008}*{Theorem~E.2.4}), we have
\begin{equation} \label{eq:dbl-ind}
\ell^2(G/V) \cong \Ind_V^G(1_V) \cong \Ind_H^G\big(\Ind_V^H(1_V)\big) \cong \Ind_H^G(\lambda_K).
\end{equation}
As $\lambda_K$ splits as a sum of irreducible representations of $K$ (and in fact all irreducible representations of $K$ occur as direct summands), we are led to consider representations of the form $\Ind_H^G(\sigma)$, where $H$ is an open subgroup of $G$ and $\sigma$ is an irreducible representation of some finite quotient of $H$. There is a general criterion known as the Mackey irreducibility criterion for determining whether representations of the form $\Ind_H^G(\sigma)$ are irreducible for $H$ an open subgroup of $G$ and $\sigma$ an irreducible \emph{finite-dimensional} representation of $H$. The criterion is usually stated for discrete groups but works equally well in this more general setting. It is due to Mackey~\cite{Mackey1951} when $\sigma$ is one-dimensional and to Corwin~\cite{Corwin1975} in the general case. Below we state and prove a special version of the criterion adapted to our situation.

If $H \leq G$, $g \in G$, and $\sigma$ is a representation of $H$, define the representation $\sigma^g$ of $H^g$ by
\[
\sigma^g(x) = \sigma(x^{g^{-1}}).
\]
\begin{prop} \label{p:ind-reps}
Let $G$ be a Roelcke precompact subgroup of $S_\infty$. Then the following hold:
\begin{enumerate} \romanenum
\item \label{i:pir:irred} If $H \leq G$ is a commensurator, $V \unlhd H$ is open, and $\sigma$ is a representation of $H/V$, then $\Ind_H^G(\sigma)$ is irreducible iff $\sigma$ is.
\item \label{i:pir:equiv} If $H_1, H_2 \leq G$ are commensurators, $V_1 \unlhd H_1$, $V_2 \unlhd H_2$ are open, and $\sigma_1$, $\sigma_2$ are irreducible representations of $H_1/V_1$, $H_2/V_2$, respectively, then $\Ind_{H_1}^G(\sigma_1) \cong \Ind_{H_2}^G(\sigma_2)$ iff there exists $g \in G$ such that $H_2 = H_1^g$ and $\sigma_2 \cong \sigma_1^g$.
\end{enumerate}
\end{prop}
\begin{proof}
\eqref{i:pir:irred}. $(\Rightarrow)$ If $\sigma = \sigma_1 \oplus \sigma_2$, then $\Ind_H^G(\sigma) = \Ind_H^G(\sigma_1) \oplus \Ind_H^G(\sigma_2)$.

$(\Leftarrow)$ Suppose $\sigma$ is irreducible and denote $\pi = \Ind_H^G(\sigma)$. We first show that
\begin{equation} \label{eq:HV}
\mcH_V(\pi) = \set{f \in \mcH(\pi) : f(x) = 0 \text{ for } x \notin H}.
\end{equation}
Suppose first that $f(x) = 0$ for all $x \notin H$. Let $g \in V$. If $x \notin H$, then $g^{-1}x \notin H$ and $0 = f(x) = f(g^{-1}x) = (\pi(g)f)(x)$. If $x \in H$, then
\[
(\pi(g)f)(x) = f(g^{-1}x) = f(xx^{-1}g^{-1}x) = \sigma(x^{-1}gx) f(x) = f(x)
\]
as $\sigma$ is trivial on $V$. For the other direction, suppose that $f \in \mcH(\pi)$ is $V$-invariant and $x \in G$ is such that $f(x) \neq 0$. If $x \notin H$, then as $H$ is its own commensurator, by \eqref{eq:defcomm}, $HxH$ contains infinitely many left cosets of $H$. Since $[H : V] < \infty$, $VxH$ also contains infinitely many left cosets of $H$. As $f$ is $V$-invariant, its norm must be infinite, contradiction. We thus obtain that
\begin{equation} \label{eq:sigma-is}
f \mapsto f(1) \text{ is an isomorphism between } \pi(H)|_{\mcH_V(\pi)} \text{ and } \sigma.
\end{equation}

Now suppose that $\pi$ is reducible, i.e. $\mcH(\pi) = \mcK \oplus \mcK^\perp$, where $\mcK$ is $\pi(G)$-invariant. As the projection onto $\mcK$ commutes with $\pi(V)$, we have
\[
\mcH_V(\pi) = (\mcH_V(\pi) \cap \mcK) \oplus (\mcH_V(\pi) \cap \mcK^\perp)
\]
and the two parts on the right-hand side are $\pi(H)$-invariant. By \eqref{eq:sigma-is} and the irreducibility of $\sigma$, either $\mcH_V(\pi) \sub \mcK$ or $\mcH_V(\pi) \sub \mcK^\perp$. Since by \eqref{eq:HV}, $\mcH_V(\pi)$ is cyclic for $\pi$, we have that $\mcK = \mcH(\pi)$ or $\mcK^\perp = \mcH(\pi)$, proving that $\pi$ is irreducible.

\eqref{i:pir:equiv}. $(\Leftarrow)$ Let $T \colon \mcH(\sigma_1) \to \mcH(\sigma_2)$ be a unitary operator that realizes the equivalence $\sigma_1^g \cong \sigma_2$ (i.e. $\sigma_2 T = T \sigma_1^g$). Let $\pi_i = \Ind_{H_i}^G(\sigma_i)$ and define the map $U \colon \mcH(\pi_1) \to \mcH(\pi_2)$ by
\[
U(f)(x) = Tf(x^{g^{-1}})
\]
It is not difficult to check that $U$ is a well-defined unitary equivalence between $\pi_1$ and $\pi_2$.

$(\Rightarrow)$ Suppose that $\pi_1$ and $\pi_2$ are equivalent. Then there exists a non-zero $f \in \mcH(\pi_1)$ which is invariant under $\pi_1(V_2)$. By the same argument as in \eqref{i:pir:irred}, we obtain that there is $g \in G$ such that $V_2 g H_1$ contains only finitely many left cosets of $H_1$, or, equivalently, $[V_2 : V_2 \cap H_1^g] < \infty$. Symmetrically, we find $h \in G$ such that $[V_1 : V_1 \cap H_2^h] < \infty$. For two subgroups $A, B \leq G$, say that \df{$A$ is large in $B$} if $[B : A \cap B] < \infty$. As $[H_i : V_i] < \infty$, we have that
\begin{align}
H_1^g &\text{ is large in } H_2, \quad \text{and} \label{eq:H1inH2} \\
H_2^h &\text{ is large in } H_1, \label{eq:H2inH1}
\end{align}
so by conjugating \eqref{eq:H1inH2} by $h$ and using transitivity, we can conclude that $H_1^{hg}$ is large in $H_1$. Applying Lemma~\ref{l:left-right}, we obtain that $H_1$ and $H_1^{hg}$ are commensurate and therefore equal (by Lemma~\ref{l:commen}~\eqref{i:lcm:2cms}). Conjugating \eqref{eq:H2inH1} by $g^{-1}h^{-1}$, we obtain that $H_2^{g^{-1}}$ is large in $H_1^{g^{-1}h^{-1}} = H_1$, while conjugating \eqref{eq:H1inH2} by $g^{-1}$, we see that $H_1$ is large in $H_2^{g^{-1}}$, so that $H_1$ and $H_2^{g^{-1}}$ are commensurate and therefore equal. So finally, $H_2 = H_1^g$.

Now let $\pi_1' = \Ind_{H_1^g}^G(\sigma_1^g) = \Ind_{H_2}^G(\sigma_1^g)$. By the $(\Leftarrow)$ direction and the hypothesis, $\pi_1' \cong \pi_1 \cong \pi_2$. Let $U \colon \mcH(\pi_1') \to \mcH(\pi_2)$ realize the equivalence. Then we must have $U(\mcH_{V_2}(\pi_1')) = \mcH_{V_2}(\pi_2)$. By \eqref{eq:sigma-is} and the fact that $U$ commutes with the $H_2$-action, we have that $\sigma_2 \cong \sigma_1^g$.
\end{proof}

We are now ready to prove the main theorem.
\begin{theorem} \label{th:main-th}
Suppose that $G$ is a Roelcke precompact subgroup of $S_\infty$. Then every unitary representation of $G$ is a sum of irreducible representations of the form $\Ind_H^G(\sigma)$, where $H$ is a commensurator and $\sigma$ is an irreducible representation of $H$ that factors through a finite quotient of $H$.
\end{theorem}
\begin{proof}
Let $\pi$ be a representation of $G$. For $\xi \in \mcH(\pi)$, let $\phi_\xi(x) = \ip{\pi(x)\xi, \xi}$ be the positive definite function on $G$ associated to $\xi$. If $\xi$ is fixed by an open subgroup $V_0 \leq G$, then by Lemma~\ref{l:const-dblcst}, $\phi_\xi$ is constant on double cosets of $V_0$, so the function $\phi_\xi$ takes only finitely many values. By Lemma~\ref{l:dense-fp}, there exists some non-zero $\xi \in \mcH(\pi)$ which is fixed by an open subgroup. Choose now a non-zero $\xi_0 \in \mcH(\pi)$ such that $\phi_{\xi_0}$ takes the \emph{minimum possible number of distinct values}. Let
\[
V = \set{g \in G : \pi(g)\xi_0 = \xi_0} = \set{g \in G : \phi_{\xi_0}(g) = \nm{\xi_0}^2}
\]
and note that as the image of $\phi_{\xi_0}$ is discrete, $V$ is open.
\begin{claim*}
If $g \notin \Comm_G(V)$, then $\phi_{\xi_0}(g) = 0$.
\end{claim*}
\begin{proof}
Let $g \notin \Comm_G(V)$ be arbitrary. By \eqref{eq:defcomm}, $VgV$ contains infinitely many left cosets of $V$. Towards a contradiction, suppose that $\ip{\pi(g)\xi_0, \xi_0} \neq 0$. Let $h_1gV, h_2gV, \ldots$ be distinct left cosets of $V$ with $h_i \in V$ for all $i$. Set $\xi_i = \pi(h_i g)\xi_0$ and let $\eta$ be a weak limit point of the $\xi_i$s. By passing to a subsequence, we can assume that $\xi_i \to^w \eta$. Since
\[
\ip{\xi_i, \xi_0} = \ip{\pi(h_i g)\xi_0, \xi_0}
    = \ip{\pi(g)\xi_0, \pi(h_i^{-1})\xi_0} = \ip{\pi(g)\xi_0, \xi_0}
\]
is bounded away from $0$, we have that $\eta \neq 0$. Next we observe that the set of values of $\phi_\eta$ is a subset of the set of values of $\phi_{\xi_0}$. Indeed, fix $i \in \N$ and note that for all $x \in G$, we have
\[ \begin{split}
\ip{\pi(x)\eta, \xi_i} &= \ip{\eta, \pi(x^{-1})\xi_i} \\
  &= \lim_{j \to \infty} \ip{\xi_j, \pi(x^{-1}h_ig)\xi_0} \\
  &= \lim_{j \to \infty} \ip{\pi(h_j g) \xi_0, \pi(x^{-1}h_ig)\xi_0} \\
  &= \lim_{j \to \infty} \phi_{\xi_0}(g^{-1}h_i^{-1} x h_j g).
\end{split} \]
Now, taking limits as $i \to \infty$,
\[
\phi_\eta(x) = \ip{\pi(x)\eta, \eta} = \lim_{i \to \infty} \ip{\pi(x)\eta, \xi_i} = \lim_{i \to \infty} \lim_{j \to \infty } \phi_{\xi_0}(g^{-1}h_i^{-1} x h_j g).
\]
As the image of $\phi_{\xi_0}$ is discrete, $\phi_\eta(x)$ is a value of $\phi_{\xi_0}$.

On the other hand, note that $\nm{\eta} < \nm{\xi_0}$. Indeed, if $\nm{\eta} = \nm{\xi_0} = \nm{\xi_i}$, then $\xi_i$ converges to $\eta$ in norm, so for all $\eps > 0$, there exists $N$ such that for $i, j > N$, $\nm{\xi_i - \xi_j} < \eps$. As $\nm{\xi_i - \xi_j}$ can take only finitely many values, the sequence $\xi_i$ is eventually constant. This contradicts the assumption that $h_i g$ and $h_j g$ are in different left cosets of $V$. It follows that the set of values of $\phi_\eta$ is a strict subset of the set of values of $\phi_{\xi_0}$ (as $\nm{\xi_0}^2$ is a value of $\phi_{\xi_0}$ which is not a value of $\phi_\eta$), contradicting the choice of $\xi_0$. This completes the proof of the claim.
\end{proof}

Put now $H = \Comm_G(V)$ and $V' = \bigcap_{h \in H} V^h$. Then $V' \unlhd H$ and as $V$ has finite index in $H$, $V'$ also has finite index in $H$. Let $\mcK = \Span\set{\pi(g)\xi_0 : g \in H}$ and note that $\mcK$ is finite-dimensional.

\begin{claim*}
$\pi(x)\mcK \perp \pi(y)\mcK$ if $xH \neq yH$.
\end{claim*}
\begin{proof}
Let $g, h \in H$. We have:
\[
\ip{\pi(xg)\xi_0, \pi(yh)\xi_0} = \ip{\pi(h^{-1}y^{-1}xg)\xi_0, \xi_0} = \phi_{\xi_0}(h^{-1}y^{-1}xg) = 0.
\]
The last equality follows from the fact that if $y^{-1}x \notin H$, then $h^{-1}y^{-1}xg \notin H$ and the previous claim.
\end{proof}

Note now that $\mcK$ is fixed pointwise by $V'$. Indeed, if $g \in H$ and $h \in V'$, by the definition of $V'$, $hgV = gV$, so there exists $h' \in V$ such that $hg = gh'$ and $\pi(h)\pi(g)\xi_0 = \pi(gh')\xi_0 = \pi(g)\xi_0$. Thus we obtain a representation of $H$ on $\mcK$ that factors through $H/V'$. Denote this representation by $\sigma$. Let $T$ be a system of left coset representatives for $H$. We verify that the partial isometry $U \colon \mcH(\Ind_H^G(\sigma)) \to \mcH(\pi)$ given by
\[
U(f) = \boplus_{x \in T} \pi(x)f(x)
\]
does not depend on the choice of $T$ and is a unitary equivalence between $\Ind_H^G(\sigma)$ and the cyclic subrepresentation of $\pi$ generated by $\xi_0$. That it does not depend on $T$ follows from \eqref{eq:induced-cond}; to check that it intertwines the representations, observe that
\[ \begin{split}
U(\Ind_H^G(\sigma)(g)f)(x) &= \boplus_{x \in T} \pi(x) f(g^{-1}x) \\
    &= \boplus_{x \in g^{-1}T} \pi(gx) f(x) \\
    &= \pi(g)\big(U(f)(x)\big).
\end{split} \]
So we obtained that $\pi$ contains a subrepresentation that is isomorphic to $\Ind_H^G(\sigma)$, where $\sigma$ factors through a finite quotient of $H$. By passing to a subrepresentation if necessary, we can assume that $\sigma$ is irreducible. Then $\Ind_H^G(\sigma)$ is irreducible by Proposition~\ref{p:ind-reps}. Now using Zorn's lemma, we conclude that $\pi$ is actually a sum of such representations.
\end{proof}

\section{Open subgroups and imaginaries} \label{s:examples}
In the previous section, we saw that in order to describe completely the representations of an automorphism group of an $\omega$-categorical structure, it suffices to understand the lattice of its open subgroups. It is most natural to understand the open subgroups in terms of the structure the group acts on. As we will see below (and as is well known), the open subgroups of the automorphism group correspond precisely to the imaginary elements of the structure. A particularly simple situation is when we can see all the open subgroups already in the structure itself, that is, when the structure \df{eliminates imaginaries} in a suitable weak sense. We proceed now with the formal definitions.

Let $\bX$ be an $\omega$-categorical structure and $G = \Aut(\bX)$. Recall that by the Ryll--Nardzewski theorem, a set $A \sub \bX^n$ is (first-order) definable iff it is $G$-invariant. We are going to use the two terms interchangeably. A tuple $\bar b \in \bX^n$ is \df{algebraic} over $\bar a \in \bX^m$ if the orbit $G_{\bar a} \cdot \bar b$ is finite. The \df{algebraic closure} of $\bar a$ is the set of all $b \in \bX$ algebraic over $\bar a$. As $\bX$ is $\omega$-categorical, the algebraic closure of a finite set is always finite. An \df{imaginary element of $\bX$} (or just an imaginary) is an equivalence class of some definable equivalence relation on a definable subset of $\bX^n$. If $\theta$ is a definable equivalence relation and $\bar a \in \bX^n$, $\bar a / \theta$ will denote the $\theta$-equivalence class of $\bar a$. If $\alpha = \bar a / \theta$ is an imaginary, we will denote by $G_\alpha$ the stabilizer of $\alpha$ in $G$. As $G_{\bar a} \leq G_\alpha$, $G_\alpha$ is open in $G$. Conversely, if $V \leq G$ is an open subgroup, there exists $\bar a$ such that $G_{\bar a} \leq V$. If we define the equivalence relation $\theta$ on $G \cdot \bar a$ by
\[
(g_1 \cdot \bar a) \eqrel{\theta} (g_2 \cdot \bar a) \iff g_1V = g_2V,
\]
then $\theta$ is $G$-invariant and $V = G_{\bar a/\theta}$.

Recall that the structure $\bX$ admits \df{weak elimination of imaginaries} if for every imaginary $\alpha$, there exists a first-order formula $\phi(\bar x, \bar y)$ such that the set
\begin{equation} \label{eq:def-im}
D(\phi, \alpha) = \big\{\bar c \in \bX^n : \alpha = \set{\bar x \in \bX^m : \phi(\bar x, \bar c)}\big\}
\end{equation}
is finite and non-empty (that is, for every imaginary, we can choose finitely many tuples to represent it). The following lemma is folklore but I have not been able to find a suitable reference. I am grateful to Martin Hils for explaining it to me.

\begin{lemma} \label{l:weak-el}
Suppose that $\bX$ is an $\omega$-categorical structure that admits weak elimination of imaginaries. Then for every open subgroup $V \leq G$, there exists a unique finite, algebraically closed substructure $\bA \sub \bX$ such that $G_{(\bA)} \leq V \leq G_{\bA}$ and $G_\bA = \Comm_G(V)$. In particular, every commensurator is of the form $G_\bA$ for some $\bA$.
\end{lemma}
\begin{proof}
By the preceding discussion, there exists an imaginary $\alpha$ such that $V = G_\alpha$. Let $\phi$ be a formula such that the set $D(\phi, \alpha)$ defined in \eqref{eq:def-im} is finite and non-empty. Let $\bA$ be the algebraic closure of $D(\phi, \alpha)$. Then clearly $G_{(\bA)} \leq V$. Also, from the definition of $D(\phi, \alpha)$, $V \cdot D(\phi, \alpha) = D(\phi, \alpha)$ and thus $V \cdot \bA = \bA$, so that $V \leq G_\bA$. The group $G_{(\bA)}$ has finite index in $G_\bA$ because it is equal to the kernel of the homomorphism $G_\bA \to \Aut(\bA)$ given by restriction. To prove that $G_\bA = \Comm_G(V)$, it suffices to check that $G_\bA$ has no proper supergroups in which it is of finite index. To see this, note that, by the definition of algebraically closed, if $g \cdot \bA \neq \bA$, then the orbit $G_\bA \cdot (g \cdot \bA)$ is infinite, showing that $G_\bA$ has infinite index in $\langle G_\bA, g \rangle$. The uniqueness of $\bA$ follows from the fact that the commensurator of $V$ is uniquely defined and two different algebraically closed substructures have different stabilizers.
\end{proof}
In fact, the converse of Lemma~\ref{l:weak-el} also holds (see \cite{Hodges1993}*{Exercise~7.3.16}) and for the examples we consider below, the conclusion of the lemma can easily be verified directly.

If the structure $\bX$ is moreover homogeneous (as defined in Section~\ref{s:olig-groups}), then for every finite substructure $\bA \sub \bX$, the canonical homomorphism $G_\bA \to \Aut(\bA)$ (that is, the restriction to $\bA$) is surjective and we have the following.

\begin{cor} \label{c:wk-elim}
Let $\bX$ be an $\omega$-categorical, homogeneous structure that admits weak elimination of imaginaries and $G = \Aut(\bX)$. Then the following is a complete list of the irreducible representations of $G$:
\begin{multline*}
\set{\Ind_{G_\bA}^G(\sigma) : \bA \sub \bX \text{ is finite, algebraically closed and } \\
    \sigma \text{ is an irreducible representation of } \Aut(\bA)}.
\end{multline*}
Moreover, this list is without repetitions (only one substructure appears of each isomorphism type).
\end{cor}
\begin{proof}
Every representation in the above list is irreducible by Proposition~\ref{p:ind-reps}. Conversely, let $\pi$ be an irreducible representation of $G$. By Theorem~\ref{th:main-th}, there exist open subgroups $V \unlhd H \leq G$ such that $H = \Comm_G(V)$ and an irreducible representation $\sigma$ of $H/V$ such that $\pi = \Ind_H^G(\sigma)$. By Lemma~\ref{l:weak-el}, there exists a finite, algebraically closed substructure $\bA \sub \bX$ such that $G_{(\bA)} \leq V$ and $G_\bA = H$. As the quotient map $G_\bA \to G_\bA/V$ factors through $G_\bA/G_{(\bA)} \cong \Aut(\bA)$, we can consider $\sigma$ as a representation of $\Aut(\bA)$. Finally, that the list is without repetitions follows from Proposition~\ref{p:ind-reps} and homogeneity (the groups $G_\bA$ and $G_\bB$ are conjugate iff $\bA$ and $\bB$ are isomorphic).
\end{proof}

It is well known and not difficult to check that Examples~\eqref{i:ex:sinfty}--\eqref{i:ex:R} from the introduction admit weak elimination of imaginaries (see \cite{Hodges1993}*{Section~4.2} for a general method to verify this), so in particular all of their irreducible representations are obtained by induction from representations of automorphism groups of finite substructures. In the case of $S_\infty$, we obtain the theorem of Lieberman~\cite{Lieberman1972}. The finite groups that appear in the representations of $S_\infty$ and $\Homeo(2^\N)$ are the symmetric groups, while the ones associated to $\GL(\infty, \F_q)$ are $\GL(n, \F_q)$. As any finite group can be realized as the automorphism group of a finite graph, we see that the representations of $\Aut(\bR)$ encode the representations of all finite groups.
\begin{remark} \label{r:Olshanski}
In \cite{Olshanski1991a}, Olshanski gives a description of the representations of another completion of the inductive limit $\varinjlim \GL(n, \F_q)$ which is different from our $\GL(\infty, \F_q)$. It is possible to obtain a proof of his result by our methods as follows. Let $\bV$ be the vector space over $\F_q$ with basis $\set{e_1, e_2, \ldots}$ and let $e_1', e_2', \ldots$ be the elements in the dual defined by $\ip{e_i, e_j'} = \delta_{ij}$. Let $\bV'$ be the (proper) subspace of the dual generated by $e_1', e_2', \ldots$. Our structure then consists of the disjoint union of the vector spaces $\bV$ and $\bV'$ together with binary relations for the pairing $\bV \times \bV' \to \F_q$. One checks that it is homogeneous and that its automorphism group is isomorphic to the one considered in \cite{Olshanski1991a}. The finite substructures of this structure are pairs of the form $(\bA, \bA')$, where $\bA$ is a finite subspace of $\bV$ and $\bA'$ is a finite subspace of $\bV'$.
\end{remark}

\begin{remark*}
The structure in Example~\eqref{i:ex:Cherlin} does not admit weak elimination of imaginaries. Another standard example of an $\omega$-categorical structure that does not eliminate imaginaries is the group $\Gamma = (\Z/4\Z)^{<\N}$; the cosets of $\Gamma/2\Gamma$ are imaginaries that cannot be eliminated.
\end{remark*}

In $(\Q, <)$, finite substructures are rigid, so we have the following.
\begin{cor} \label{c:repsQ}
The irreducible representations of $\Aut(\Q)$ are
\[
\Aut(\Q) \actson \ell^2(\Q^{[n]}), \quad n \in \N.
\]
\end{cor}
As $\Aut(\Q)$ can be densely embedded in $\Homeo^+(\R)$ (the latter being equipped with any of the uniform convergence, pointwise convergence, or compact-open topologies which coincide on it), we obtain the result of Megrelishvili~\cite{Megrelishvili2001} mentioned in the introduction.
\begin{cor} \label{c:HomeoR}
The group $\Homeo^+(\R)$ has no non-trivial unitary representations.
\end{cor}
\begin{proof}
Let $G = \Homeo^+(\R)$ and let $\pi$ be a unitary representation of $G$. Let
\[
H = \set{g \in G : g(\Q) = \Q}
\]
(here $\Q$ is regarded as a subset of $\R$) and note that $H$ is a continuous homomorphic image of $\Aut(\Q)$. By Corollary~\ref{c:repsQ}, $\pi|_H$ is a direct sum $\boplus_i \ell^2(\Q^{[n_i]})$. Define $g_k \in H$ by $g_k(x) = x + 1/k$. Then $g_k$ converges to the identity in $G$ but assuming that $n_i \neq 0$ for some $i$ and letting $\delta_\bA$ be the vector in $\ell^2(\Q^{[n_i]})$ which is $1$ on some subset $\bA \in \Q^{[n_i]}$ and $0$ everywhere else, we obtain $\nm{\pi(g_k)\delta_\bA - \delta_\bA} = \sqrt{2}$ for all $k$, a contradiction with the continuity of $\pi$. Hence, $n_i = 0$ for all $i$ and $\pi$ is trivial on $H$. As $H$ is dense in $G$, $\pi$ must be trivial on $G$, too.
\end{proof}
\begin{remark*}
In \cite{Megrelishvili2001}, Megrelishvili proves the stronger result that $\Homeo^+(\R)$ does not have any non-trivial representations by isometries on reflexive Banach spaces.
\end{remark*}

We finally remark that the continuity assumption in our definition of a representation is often not restrictive because of the phenomenon of automatic continuity. Say that a Polish group $G$ satisfies the \df{automatic continuity property} if every homomorphism $f \colon G \to H$ to a separable group $H$ is continuous. This quite remarkable property has been verified for many non-locally compact groups, including Examples~\eqref{i:ex:sinfty}--\eqref{i:ex:R} from the introduction (for $S_\infty$, $\GL(\infty, \F_q)$, and $\Aut(\bR)$ it is due to Kechris--Rosendal~\cite{Kechris2007a} and Hodges--Hodkinson--Lascar--Shelah~\cite{Hodges1993a}; for $\Aut(\Q)$ and $\Homeo(2^\N)$, it is a result of Rosendal--Solecki~\cite{Rosendal2007}).
\begin{cor}
Let $G$ be an oligomorphic group that satisfies the automatic continuity property. Then the conclusion of Theorem~\ref{th:main-th} applies to any (not a priori continuous) unitary representation of $G$ on a \emph{separable} Hilbert space.
\end{cor}
Of course, the condition that the Hilbert space be separable is necessary: every oligomorphic group considered as discrete has its left-regular representation which is not of the kind described in the theorem.

\section{Property (T)} \label{s:propT}
We recall the definition of property (T) for topological groups.
\begin{defn} \label{df:propT}
Let $G$ be a group, $Q \sub G$, $\eps > 0$. If $\pi$ is a unitary representation of $G$, we say that a non-zero vector $\xi \in \mcH(\pi)$ is \df{$(Q, \eps)$-invariant for $\pi$} if for all $x \in Q$, $\nm{\pi(x)\xi - \xi} \leq \eps \nm{\xi}$. The topological group $G$ is said to have \df{Kazhdan's property (T)} if there exist a compact $Q \sub G$ and $\eps > 0$ such that every representation $\pi$ of $G$ that has a $(Q, \eps)$-invariant vector, actually has an invariant vector. $G$ has the \df{strong property (T)} if $Q$ can be chosen to be finite. The set $Q$ is called a \df{Kazhdan set} for $G$.
\end{defn}
This property has mostly been studied for locally compact groups, where it has found many applications, but by now, there are also some non-locally compact examples. Of course, groups that have no unitary representations trivially have property (T) but there are also some large groups that have property (T) for non-trivial reasons. The first examples of this type were the loop groups over $\SL(n, \C)$ (Shalom~\cite{Shalom1999}), another is the infinite-dimensional unitary group (Bekka~\cite{Bekka2003}).

We note that property (FH), which is equivalent to (T) for locally compact Polish groups (see \cite{Bekka2008}), is strictly weaker in general. While all Roelcke precompact groups have property (FH) (Proposition~\ref{p:OB}), additional work is needed to find appropriate Kazhdan sets.

As every representation of a Roelcke precompact subgroup of $S_\infty$ splits as a sum of irreducibles, to verify that such a group has property (T), it suffices to find $(Q, \eps)$ such that $\sup_{g \in Q} \nm{\pi(g)\xi - \xi} \geq \eps$ for every non-trivial, irreducible $\pi$ and unit vector $\xi \in \mcH(\pi)$.

We have seen in the previous section that when an $\omega$-categorical structure $\bX$ admits weak elimination of imaginaries, all irreducible representations of $G = \Aut(\bX)$ can be extracted from the action $G \actson \bX$. As we are concerned with non-trivial representations, it will be convenient to disregard the points in $\bX$ that are fixed by $G$. Denote $\bX_0 = \bX \sminus \set{a \in \bX : G \cdot a = a}$. The next proposition shows that to verify property (T), it suffices to consider only tensor powers of the representation $\ell^2(\bX_0)$.
\begin{prop} \label{p:subreps}
Let $\bX$ be an $\omega$-categorical structure that admits weak elimination of imaginaries and $G = \Aut(\bX)$. Then every non-trivial, irreducible representation of $G$ is a subrepresentation of $\ell^2(\bX_0^n)$ for some $n > 0$.
\end{prop}
\begin{proof}
Let $\pi$ be an irreducible representation of $G$. In the same way as in the proof of Corollary~\ref{c:wk-elim}, we see that there exists an algebraically closed, finite substructure $\bA \sub \bX$ such that $\pi$ is equivalent to $\Ind_{G_\bA}^G(\sigma)$ for some irreducible representation $\sigma$ of the finite group $K = G_\bA/G_{(\bA)}$. As every irreducible representation of a finite group is contained in its left-regular representation, in the same way as in \eqref{eq:dbl-ind}, we obtain:
\[
\pi \cong \Ind_{G_\bA}^G(\sigma) \leq \Ind_{G_\bA}^G(\lambda_K) \cong \ell^2(G/G_{(\bA)}).
\]
If $k = |\bA|$, then $G/G_{(\bA)}$ is an orbit of the action $G \actson \bX^k$, so $\ell^2(G/G_{(\bA)}) \sub \ell^2(\bX^k)$. To finish the proof, we check by induction on $k$ that $\ell^2(\bX^k)$ splits as a direct sum $\boplus_{i = 1}^s \ell^2(\bX_0^{n_i})$. If $k = 0$, we can take $s = 1$ and $n_1 = 0$. Denote $Y = \bX \sminus \bX_0$ and observe that as the action $G \actson \bX$ has only finitely many orbits, $Y$ is finite. Now we deduce the statement for $k+1$ from the one for $k$:
\[ \begin{split}
\ell^2(\bX^{k+1}) \cong \ell^2(\bX) \otimes \ell^2(\bX^k)
    &\cong \big(\ell^2(\bX_0) \oplus \ell^2(Y)\big) \otimes \big(\bigoplus_{i=1}^s \ell^2(\bX_0^{n_i})\big) \\
    &\cong \bigoplus_{i=1}^s \ell^2(\bX_0^{n_i+1}) \oplus |Y| \cdot \bigoplus_{i=1}^s \ell^2(\bX_0^{n_i}).
\end{split} \]
Finally, as $\pi$ is irreducible and non-trivial, it is a subrepresentation of one of the $\ell^2(\bX_0^{n_i})$ for some $n_i > 0$.
\end{proof}

If $Y$ is a set, denote by $\ell^1_+(Y)$ the subset of $\ell^1(Y)$ consisting of all non-negative $\ell^1$ functions of norm $1$.
\begin{lemma} \label{l:nonam-prod}
Suppose that a group $G$ acts on a set $Y$ so that there exists a subset $Q \sub G$ and $\eps > 0$ such that for every $f \in \ell^1_+(Y)$,
\begin{equation} \label{eq:def-nonamen}
\sup_{g \in Q} \nm{g \cdot f - f}_1 \geq \eps.
\end{equation}
Then \eqref{eq:def-nonamen} holds also for every $f \in \ell^1_+(Y^n)$ (with the diagonal action $G \actson Y^n$) for every $n \geq 2$.
\end{lemma}
\begin{proof}
The proof is by induction on $n$. Let $f \in \ell^1_+(Y^{n+1})$. Define $\tilde f \in \ell^1_+(Y^n)$ by $\tilde f(\bar x) = \sum_{y \in Y} f(\bar x, y)$. Then for any $g \in G$,
\[ \begin{split}
\nm{g \cdot \tilde f - \tilde f}_1 &= \sum_{\bar x} |\tilde f(g^{-1} \cdot \bar x) - \tilde f(\bar x)| \\
&= \sum_{\bar x} \Big|\sum_y f(g^{-1} \cdot \bar x, y) - \sum_y f(\bar x, y)\Big| \\
&= \sum_{\bar x} \Big|\sum_y f(g^{-1} \cdot \bar x, g^{-1} \cdot y) - \sum_y f(\bar x, y)\Big| \\
&\leq \sum_{\bar x} \sum_y |f(g^{-1} \cdot \bar x, g^{-1} \cdot y) - f(\bar x, y)| = \nm{g \cdot f - f}_1,
\end{split} \]
showing that \eqref{eq:def-nonamen} holds for $f$ if it holds for $\tilde f$.
\end{proof}

\begin{prop} \label{p:KazhSet}
Suppose that $\bX$ is an $\omega$-categorical structure that admits weak elimination of imaginaries and $G = \Aut(\bX)$. Then if $Q \sub G$ is compact and
\begin{equation} \label{eq:equivT}
\inf_{f \in \ell^1_+(\bX_0)} \sup_{g \in Q} \nm{g \cdot f - f}_1 > 0,
\end{equation}
$Q$ is a Kazhdan set for $G$. Conversely, if the action $G \actson \bX_0$ has only infinite orbits and $Q$ is a Kazhdan set for $G$, then \eqref{eq:equivT} holds.
\end{prop}
\begin{proof}
Suppose first that \eqref{eq:equivT} holds and set $\eps = \inf_{f \in \ell^1_+(\bX_0)} \sup_{g \in Q} \nm{g \cdot f - f}_1$. By Proposition~\ref{p:subreps}, to see that $Q$ is a Kazhdan set for $G$, it suffices to show that for every $n > 0$ and every $f \in \ell^2(\bX_0^n)$, $\sup_{g \in Q} \nm{g \cdot f - f}_2 \geq (\eps/2)\nm{f}_2$. Suppose that $f \in \ell^2(\bX_0^n)$, $\nm{f}_2 = 1$, $g \in G$. Set $\tilde f(x) = |f(x)|^2$. Then $\tilde f \in \ell^1_+(\bX_0^n)$ and we have
\[ \begin{split}
\nm{g \cdot \tilde f - \tilde f}_1 &= \sum_{\bar x \in \bX_0^n} \big| |f(g^{-1} \cdot \bar x)|^2 - |f(\bar x)|^2 \big| \\
& = \sum_{\bar x \in \bX_0^n} \big||f(g^{-1} \cdot \bar x)| - |f(\bar x)|\big| \cdot (|f(g^{-1} \cdot \bar x)| + |f(\bar x)|) \\
&= \ip{\big| g \cdot |f| - |f|\big|, g \cdot |f| + |f|} \\
&\leq \nm{g \cdot f - f}_2 \nm{g \cdot |f| + |f|}_2 \leq 2\nm{g \cdot f - f}_2.
\end{split} \]
Combining this with Lemma~\ref{l:nonam-prod} finishes the proof.

Now suppose that \eqref{eq:equivT} does not hold but there exists $\eps > 0$ such that $(Q, \eps)$ is a Kazhdan pair for $G$. Then there exists $f \in \ell^1_+(\bX_0)$ such that $\sup_{g \in Q} \nm{g \cdot f - f}_1 < \eps^2$. Using the inequality $|a-b|^2 \leq |a^2 - b^2|$, which holds for all non-negative real numbers $a$ and $b$, we see that $f^{1/2}$ is a $(Q, \eps)$-invariant vector for the representation $\ell^2(\bX_0)$. By property (T), $\ell^2(\bX_0)$ has an invariant vector $f_0$ and then $\set{a \in \bX_0 : |f_0(a)| = \max |f_0|}$ is a finite $G$-invariant set in $\bX_0$, which is a contradiction with the hypothesis that the action $G \actson \bX_0$ has infinite orbits.
\end{proof}

The next lemma shows that at least for certain well-behaved structures, we can always construct a compact set $Q \sub G$ that satisfies \eqref{eq:equivT}. Recall that an $\omega$-categorical structure has \df{no algebraicity} if the algebraic closure of every finite substructure $\bA$ is $\bA$ itself. This is equivalent to the condition that the stabilizer $G_{(\bA)}$ has infinite orbits on $\bX \sminus \bA$.
\begin{lemma} \label{l:consrQ}
Let $\bX$ be a homogeneous, relational structure with no algebraicity and let $G = \Aut(\bX)$. Then there exists a compact set $Q \sub G$ such that for every $f \in \ell^1_+(\bX)$,
\[
\sup_{g \in Q} \nm{g \cdot f - f}_1 \geq 1/2.
\]
\end{lemma}
\begin{proof}
We will use a back-and-forth construction. Enumerate $\bX = \set{a_1, a_2, \ldots}$ and set $\bA_n = \set{a_i : i \leq n}$ for $n \geq 0$. We will define inductively finite families $S_1, S_2, \ldots$ of finite partial automorphisms of $\bX$ with the following properties:
\begin{enumerate}
\item $S_1 = \set{\emptyset}$;
\item \label{i:unique-rest} if $\phi \in S_{n+1}$, there exists a unique $\phi' \in S_n$ such that $\phi \supseteq \phi'$;
\item \label{i:dom-ran} for every $\phi \in S_{2n}$, $\bA_n \sub \dom \phi$ and for every $\phi \in S_{2n+1}$, $\bA_n \sub \ran \phi$;
\item \label{i:inter} for every $\phi \in S_{2n} \cup S_{2n+1}$, $\dom \phi \cap \ran \phi \sub \bA_n$;
\item \label{i:cond-odd} for every $\phi \in S_{2n-1}$, there exist $\psi_1, \psi_2, \ldots, \psi_{2^{n+1}} \in S_{2n}$ such that $\psi_i \supseteq \phi$ for every $i$ and $\ran \psi_i \cap \ran \psi_j = \ran \phi$ for all $i \neq j$;
\item \label{i:cond-even} for every $\phi \in S_{2n}$, there exist $\psi_1, \psi_2, \ldots, \psi_{2^{n+1}} \in S_{2n+1}$ such that $\psi_i \supseteq \phi$ for every $i$ and $\dom \psi_i \cap \dom \psi_j = \dom \phi$ for all $i \neq j$.
\end{enumerate}
One can view the sets $S_n$ as the levels of a finitely splitting rooted tree.

Recall that for a homogeneous structure the no algebraicity assumption can be reformulated as the following extension property (see Hodges~\cite{Hodges1993}):
\begin{quotation}
for every pair of finite structures $\bA, \bB$, embeddings $\psi \colon \bA \to \bX$ and $\theta \colon \bA \to \bB$, and finite set $D \sub \bX$, there exists an embedding $\phi \colon \bB \to \bX$ such that $\phi \circ \theta = \psi$ and $\phi(\bB) \cap D \sub \psi(\bA)$.
\end{quotation}

Start the construction with $S_1 = \set{\emptyset}$. Suppose now that $S_{2n-1}$ has been constructed and proceed to build $S_{2n}$. For every $\phi \in S_{2n-1}$, we will construct a set $E_\phi$ of extensions of $\phi$ to put in $S_{2n}$. If $a_n \in \dom \phi$, set $E_\phi = \set{\phi}$. If $a_n \notin \dom \phi$, let $\bB = \dom \phi \cup \set{a_n}$ and construct $E_\phi = \set{\psi_1, \psi_2, \ldots, \psi_{2^{n+1}}}$ inductively as follows. Suppose $\psi_1, \ldots, \psi_k$ have been constructed and apply the extension property to the inclusion map $\dom \phi \to \bB$ and the embedding $\phi \colon \dom \phi \to \bX$ to obtain a partial automorphism $\psi_{k+1}$ with domain $\bB$ that extends $\phi$ and such that
\begin{equation} \label{eq:disj}
\ran \psi_{k+1} \cap \Big(\bigcup_{i \leq k} \ran \psi_i \cup \bB \Big) \sub \ran \phi.
\end{equation}
Finally, set $S_{2n} = \bigcup_{\phi \in S_{2n-1}} E_\phi$. To construct $S_{2n+1}$ apply the same procedure symmetrically (replacing $\phi$ with $\phi^{-1}$). Conditions \eqref{i:dom-ran}, \eqref{i:cond-odd}, \eqref{i:cond-even}, and the existence part of \eqref{i:unique-rest} are satisfied by construction. For the uniqueness part of \eqref{i:unique-rest}, note that any two distinct elements of $S_n$ are incomparable. Finally, suppose that condition \eqref{i:inter} is not verified and let $i$ be the least natural number such that there exists $\phi_i \in S_i$ and  $a \in (\dom \phi_i \cap \ran \phi_i) \sminus \bA_{\lfloor i/2\rfloor}$. We will consider the case when $i = 2n$ is even, the other one being similar. Denote by $\phi_{2n-1}$ the element in $S_{2n-1}$ such that $\phi_{2n-1} \sub \phi_{2n}$. By \eqref{eq:disj}, $\ran \phi_{2n} \cap \dom \phi_{2n} \sub \ran \phi_{2n-1}$, so $a \in \ran \phi_{2n-1}$. On the other hand, by construction, $\dom \phi_{2n} = \dom \phi_{2n-1} \cup \set{a_n}$. Hence $a \in \dom \phi_{2n-1}$, a contradiction with the minimality of $i$.

Now $\set{S_n}_n$ forms naturally an inverse system with the maps
\[
S_{n+1} \to S_n, \quad \phi \mapsto \text{the unique } \phi' \in S_n \text{ such that } \phi' \sub \phi.
\]
Denote by $Q$ the inverse limit, i.e.
\[
Q = \set{g \in G : \forall n \exists \phi \in S_n \ \phi \sub g}.
\]
(Every map in the inverse limit is a full automorphism of $\bX$ because of condition \eqref{i:dom-ran}.) As the sets $S_n$ are finite, $Q$ is compact.

We now check that for any $f \in \ell^1_+(\bX)$, there exists $g \in Q$ such that $\nm{g \cdot f - f}_1 \geq 1/2$. We build inductively a sequence $\phi_1, \phi_2, \ldots$ such that $\phi_i \in S_i$ and $\phi_i \sub \phi_{i+1}$. Let $\phi_1 = \emptyset$. Suppose that $\phi_{2n-1}$ is already chosen. If $a_n \in \dom \phi_{2n-1}$, set $\phi_{2n} = \phi_{2n-1}$. If $a_n \notin \dom \phi_{2n-1}$, by \eqref{i:cond-odd} and the fact that $\nm{f}_1 = 1$, there exists $\phi_{2n} \in S_{2n}$ such that
\begin{equation} \label{eq:phi-l}
f(\phi_{2n}(a_n)) \leq 2^{-(n+1)}.
\end{equation}
Now we proceed to choose $\phi_{2n+1}$. If $a_n \notin \dom \phi_{2n-1}$, let $\phi_{2n+1} \in S_{2n+1}$ be an arbitrary extension of $\phi_{2n}$. If $a_n \in \dom \phi_{2n-1}$, then by \eqref{i:inter}, $a_n \notin \ran \phi_{2n-1}$, so $a_n \notin \ran \phi_{2n}$ either (as in this case, $\phi_{2n} = \phi_{2n-1}$). Now, using \eqref{i:cond-even}, we can choose $\phi_{2n+1} \in S_{2n+1}$ so that
\begin{equation} \label{eq:phi-r}
f(\phi_{2n+1}^{-1}(a_n)) \leq 2^{-(n+1)}.
\end{equation}
In summary, we obtained that for every $n$ at least one of the inequalities \eqref{eq:phi-l} and \eqref{eq:phi-r} holds. Let $g = \bigcup_i \phi_i$. Combining \eqref{eq:phi-l} and \eqref{eq:phi-r}, we have that for every $n$,
\[
\min \big( f(a_n), f(g^{-1} \cdot a_n) \big) \leq 2^{-(n+1)}.
\]
Now a calculation yields:
\[ \begin{split}
\nm{g \cdot f - f}_1 &= \sum_{n=1}^\infty |f(g^{-1} \cdot a_n) - f(a_n)| \\
    &= \sum_{n=1}^\infty \max \big( f(a_n), f(g^{-1} \cdot a_n) \big)
        - \min \big( f(a_n), f(g^{-1} \cdot a_n) \big)\\
    &\geq \nm{f}_1 - \sum_{n=1}^\infty \min \big( f(a_n), f(g^{-1} \cdot a_n) \big) \\
    &\geq 1 - \sum_{n = 1}^\infty 2^{-(n+1)} = 1/2,
\end{split} \]
finishing the proof.
\end{proof}


Combining everything we have so far, we obtain the following.
\begin{theorem} \label{th:propT-gen}
Let $\bX$ be an $\omega$-categorical, relational, homogeneous structure with no algebraicity that admits weak elimination of imaginaries. Then $\Aut(\bX)$ has property (T).
\end{theorem}
\begin{proof}
Follows from Proposition~\ref{p:KazhSet} and Lemma~\ref{l:consrQ}.
\end{proof}

As the next theorem shows, in many concrete examples, it is not difficult to find \emph{finite} Kazhdan sets. In fact, I do not know whether, in the setting of Theorem~\ref{th:propT-gen}, it is always possible to do so (cf. Question~\eqref{qi:non-amen}).
\begin{theorem} \label{th:Tconcrete}
All of the following groups have a Kazhdan set with two elements:
\[
S_\infty, \ \Aut(\Q), \ \GL(\infty, \F_q), \ \Homeo(2^\N), \ \Aut(\bR).
\]
\end{theorem}
\begin{proof}
In view of Proposition~\ref{p:KazhSet}, to prove that the automorphism group $G$ of an $\omega$-categorical structure $\bX$ with weak elimination of imaginaries admits a finite Kazhdan set, it suffices to find a finitely-generated subgroup $\Gamma \leq G$ such that the action $\Gamma \actson \bX_0$ is \df{non-amenable} (i.e. \eqref{eq:equivT} holds with $Q$ a generating set for $\Gamma$). In fact, for all of the above groups we will find a copy of the free group $\bF_2$ in $G$ that acts freely on $\bX_0$. As it is well-known that any free action of $\bF_2$ is non-amenable, this will complete the proof.

$S_\infty$. Consider the left action of $\bF_2$ on itself. This gives a copy of $\bF_2$ in the group of all permutations of $\bF_2$ that acts freely.

$\Aut(\Q)$. It suffices to find an ordering on the group $\bF_2$ which is left-invariant and isomorphic to $\Q$. Then the left action of $\bF_2$ on itself will produce the required embedding $\bF_2 \hookrightarrow \Aut(\Q)$. It is well known that $\bF_2$ admits a \df{bi-invariant} (invariant under multiplication on both sides) linear ordering (see, for example, \cite{Rolfsen2001u}). We now check that any bi-invariant ordering on $\bF_2$ is dense without endpoints. First, if $x > 1$, then $x^2 > x$ and if $x < 1$, $x^2 < x$, showing that the ordering has no endpoints. To see that it is dense, it suffices, for every $x > 1$, to find $z$ such that $1 < z < x$. By bi-invariance, all conjugates of $x$ are $> 1$. Let now $y$ be an arbitrary element that does not commute with $x$. If $x^y < x$, we are done and if $x^y > x$, by conjugating with $y^{-1}$, we obtain that $x^{y^{-1}} < x$.

$\GL(\infty, \F_q)$. Label a basis of the vector space by the elements of $\bF_2$ and let $\bF_2$ act freely on this basis in the natural way. This action extends to an action on the whole vector space. As $\bF_2$ is torsion-free, no finite, non-empty subset of $\bF_2$ is invariant under any non-identity element of the group, so the support of every non-zero vector is moved by every non-trivial element of $\bF_2$, showing that the action is free (when restricted to the non-zero elements of the vector space).

$\Homeo(2^\N)$. Consider the natural shift action $\bF_2 \actson 2^{\bF_2}$ given by $(g \cdot \omega)(h) = \omega(g^{-1}h)$. Since this action is by homeomorphisms, it gives an action $\bF_2 \actson \Clopen(2^{\bF_2})$. Using a similar argument to the one above, we see that its restriction to $\Clopen(2^{\bF_2}) \sminus \set{\emptyset, 2^{\bF_2}}$ is free.

$\Aut(\bR)$. Consider a \df{random right Cayley graph} of the free group constructed in the following way. Let $S$ be a random set of unordered pairs $\set{g, g^{-1}}$ of elements of $\bF_2 \sminus \set{1}$, i.e. each pair is included or not in $S$ independently with probability $1/2$. The vertices of the graph are the elements of $\bF_2$ and two vertices $x, y \in \bF_2$ are connected by an edge iff $x^{-1}y \in S$. By Cameron~\cite{Cameron2000}, the random Cayley graph of $\bF_2$ is isomorphic to $\bR$ with probability $1$, so in particular, Cayley graphs of $\bF_2$ isomorphic to $\bR$ exist. As $\bF_2$ acts freely by isomorphisms on any of its Cayley graphs, we obtain the desired embedding $\bF_2 \hookrightarrow \Aut(\bR)$.
\end{proof}

\begin{remark*} It is also possible, using the method of Bekka~\cite{Bekka2003}, to find optimal Kazhdan constants for those Kazhdan sets. In fact, the constant is the same as in \cite{Bekka2003}.
\end{remark*}

\begin{remark*}
We note that not all oligomorphic groups have the strong property (T). Indeed, the group in Example~\eqref{i:ex:Cherlin} admits $(\Z / 2\Z)^\N$ as a quotient and the latter does not admit a finite Kazhdan set by \cite{Bekka2003}*{Proposition~5} (this can also easily be seen directly).
\end{remark*}

\begin{remark*}
As was already noted in \cite{Bekka2003}, there is no special connection between property (T) and amenability for non-locally compact groups. Of the groups in Theorem~\ref{th:Tconcrete}, all are amenable except $\Homeo(2^\N)$. (A topological group is called \df{amenable} if every time it acts continuously on a compact space, there is an invariant measure.)
\end{remark*}

\begin{remark*}
It is an open problem whether there exists a discrete subgroup of $\Aut(\Q)$ with property (T). (A discrete group embeds in $\Aut(\Q)$ iff it acts faithfully by orientation-preserving homeomorphisms on the reals iff it is left-orderable; see Morris~\cite{Morris1994}). I am grateful to the referee for pointing this out.
\end{remark*}

We conclude with two open problems.

\smallskip \noindent \textbf{Questions:}
\begin{enumerate}
\item Does every closed oligomorphic subgroup of $S_\infty$ have property (T)? More generally, does every Roelcke precompact Polish group have property (T)?


\item \label{qi:non-amen} Suppose that $\bX$ is an $\omega$-categorical, relational, homogeneous structure with no algebraicity. Is it always the case that the action $\Aut(\bX) \actson \bX$ is non-amenable, i.e. must there always exist a \emph{finite} set $Q \sub \Aut(\bX)$ such that for every $f \in \ell^1_+(\bX)$,
\[
\sup_{g \in Q} \nm{g \cdot f - f}_1 > 0?
\]
\end{enumerate}

\bibliography{mybiblio}
\end{document}